\documentclass[12pt,leqno,draft]{article}
\usepackage{amsfonts}
\usepackage{amsmath, amsthm, amsfonts, amssymb, color}
\usepackage{mathrsfs}
\usepackage{color}
\usepackage{stmaryrd}
\newtheorem{thm}{Theorem}[section]

\newtheorem{lem}[thm]{Lemma}

\theoremstyle{definition}
\newtheorem{defn}{Definition}[section]
\newcommand{\scr}[1]{\mathscr #1}
\definecolor{wco}{rgb}{0.5,0.2,0.3}

\numberwithin{equation}{section} \theoremstyle{remark}

\newcommand{\ua}{\uparrow}

\title{{\bf Moderate Deviation Principles for Unbounded Additive Functionals of Distribution Dependent SDEs} }
\author{
{\bf     Panpan Ren $^{2,a)}$, Shen Wang $^{1,b)}$   }\\
\footnotesize{  1)Center for Applied Mathematics, Tianjin University, Tianjin 300072, China}\\
\footnotesize{  2)Mathematic Department, City University of Hong Kong, Hong Kong, China}\\
\footnotesize{ $^{a)}$rppzoe@gmail.com, $^{b)}$wswangshen@tju.edu.cn}}
\begin{document}
\allowdisplaybreaks
\def\R{\mathbb R}  \def\ff{\frac} \def\ss{\sqrt} \def\B{\mathbf
B}
\def\N{\mathbb N} \def\kk{\kappa} \def\m{{\bf m}}
\def\ee{\varepsilon}\def\ddd{D^*}
\def\dd{\delta} \def\DD{\Delta} \def\vv{\varepsilon} \def\rr{\rho}
\def\<{\langle} \def\>{\rangle} \def\GG{\Gamma} \def\gg{\gamma}
  \def\nn{\nabla} \def\pp{\partial} \def\E{\mathbb E}
\def\d{\text{\rm{d}}} \def\bb{\beta} \def\aa{\alpha} \def\D{\scr D}
  \def\si{\sigma} \def\ess{\text{\rm{ess}}}
\def\beg{\begin} \def\beq{\begin{equation}}  \def\F{\scr F}
\def\Ric{\text{\rm{Ric}}} \def\Hess{\text{\rm{Hess}}}
\def\e{\text{\rm{e}}} \def\ua{\underline a} \def\OO{\Omega}  \def\oo{\omega}
 \def\tt{\tilde} \def\Ric{\text{\rm{Ric}}}
\def\cut{\text{\rm{cut}}} \def\P{\mathbb P} \def\ifn{I_n(f^{\bigotimes n})}
\def\C{\scr C}   \def\G{\scr G}   \def\aaa{\mathbf{r}}     \def\r{r}
\def\gap{\text{\rm{gap}}} \def\prr{\pi_{{\bf m},\varrho}}  \def\r{\mathbf r}
\def\Z{\mathbb Z} \def\vrr{\varrho} \def\ll{\lambda}
\def\L{\scr L}\def\Tt{\tt} \def\TT{\tt}\def\II{\mathbb I}
\def\i{{\rm in}}\def\Sect{{\rm Sect}}  \def\H{\mathbb H}
\def\M{\scr M}\def\Q{\mathbb Q} \def\texto{\text{o}} \def\LL{\Lambda}
\def\Rank{{\rm Rank}} \def\B{\scr B} \def\i{{\rm i}} \def\HR{\hat{\R}^d}
\def\to{\rightarrow}\def\l{\ell}\def\iint{\int}
\def\EE{\scr E}\def\no{\nonumber}
\def\A{\scr A}\def\V{\mathbb V}\def\osc{{\rm osc}}
\def\BB{\scr B}\def\Ent{{\rm Ent}}\def\3{\triangle}\def\H{\scr H}
\def\U{\scr U}\def\8{\infty}\def\1{\lesssim}\def\HH{\mathrm{H}}
 \def\T{\scr T}
\maketitle

\begin{abstract}
By comparing the original equations with the corresponding stationary ones, the moderate deviation principle (MDP) is established for unbounded additive functionals of several different models of distribution dependent SDEs, with non-degenerate and degenerate noises.
\end{abstract} \noindent
 \vskip 2cm

\section{Introduction}

To characterise long time behaviours of stochastic systems, various limit theorems, including LLN(law of large numbers), CLT(central limit theorems), and LDP(large deviation principle) have been intensively investigated in the literature of Markov processes and random sequences, see for instance \cite{B88, BM78, BM80, BWY20, C91, CDR17, Gao17, IN64, WZ19, Wu95, Wu00}. On the other hand, less is known for limit theorems on nonlinear systems, where a typical model is the distribution dependent SDE (also called McKean-Vlasov or Mean-filed SDE), which arises from characterizations on nonlinear Fokker-Planck equations and mean-filed particle systems, see \cite{HRW} and references within.  Recently, the Donsker-Varadhan LDP for path-distribution dependent SDEs was investigated in \cite{RW20} for empirical measures of distribution dependent SDEs, which in particular implies LDP for bounded continuous additive functionals. In this paper, we investigate the MDP(moderate deviation principle) for unbounded additive functionals.

Below, we first briefly recall the notion of LDP and MDP, then introduce the model studied in the present paper.

Let $E$ be a polish space, and let $(X_t)_{t\geq 0}$ be a right continuous Markov process on $E$ with infinitesimal generator $L$. For a measurable space $(E,\B)$, let $\mathscr{P}(E)$ denote the set of all probability measures on $E$ with weak topology.
Consider the empirical measures of $(X_t)_{t\geq0}$:
\[
L_t:= \frac 1 t \int_0^t \delta_{X_s} \d s, \,\,\,t > 0.
\]
The following Donsker-Varadhan type long time LDP for $L_t$ has been studied in \cite{DV75}:
\begin{align}\label{PLt}
\mathbb{P}_x(L_t\in M)\approx \exp\{ -t \inf_{\nu\in M} J(\nu)\},\quad M\subset \mathscr{P}(E),
\end{align}
where $J(\nu)= \sup_{\inf_x U(x)> 0,U\in \mathcal{D}(L)} \int \frac{-LU}{U} \d \nu$, $\mathbb{P}_x$ denotes the probability of Markov process starting from $x$, and $\mathscr{P}(E)$ is the class of all probability measures on $E$, equipped with the weak topology.

In general, when the Markov process $(X_t)_{t\geq0}$ is ergodic, in order to describe the convergence of the empirical distribution $L_t$ to the unique invariant probability measure $\bar{\mu}$ as $t\rightarrow\infty$, a standard way is to look at the convergence rate of
\[
L_t^A:=\int_E A \d L_t= \frac 1 t \int_0^t A(X_s) \d s\rightarrow \bar{\mu}(A)\,\,\,\textrm{as}\,\,\,t\rightarrow\infty
\]
for $A$ in a class of reference functions. This leads to the study of the LDP (MDP) for the additive functional $L_t^A$. When $A$ is bounded and continuous, (\ref{PLt}) and the Contraction Principle imply the LDP of $L_t^A$, that is, for $M\in \B(\mathbb{B})$,
\[
\mathbb{P}_x( L_t^A \in M)\approx \exp\{ -t \inf_{z\in M} J^A(z)\},\quad A\subset \mathbb{R},
\]
where $J^A(z)= \inf\{J(\nu);\,\,\,\int A \d\nu=z \}$. But this approach does not apply when $A$ is unbounded. So, we consider the MDP (moderate deviation principle) for $L_t^A$ with unbounded $A$, which is equivalent to LDP for the modified additive functional
\[
l_t^A:=\frac {t} {a(t)}\Big( L_t^A-\bar{\mu}(A) \Big)=\frac 1 {a(t)} \int_0^t \big(A(X_s)-\bar{\mu}(A)\big) \d s,
\]
where $a(t)$ is a positive function satisfying
\begin{align}\label{at}
\lim_{t\rightarrow\infty}\frac{\sqrt{t}}{a(t)}= 0,\,\,\,\,\lim_{t\rightarrow\infty}\frac{a(t)}{t}= 0.
\end{align}

\begin{defn}
\begin{enumerate}
\item[\bf{(1)}] $L_t^A$ is said to satisfy the upper bound uniform MDP with a rate function $I$, denoted by $L_t^A \in MDP_u(I)$, if for any $a$ satisfying (\ref{at}),
   \[
   \limsup_{t\rightarrow\infty} \frac {t}{a^2(t)} \log \mathbb{P}(l_t^A\in F)\leq -\inf_F I,\quad F\subset \mathbb{R}\,\,\, \textrm{is}\,\,\, \textrm{closed}.
   \]
\item[\bf{(2)}] $L_t^A$ is said to satisfy the lower bound uniform MDP with a rate function $I$, denoted by $L_t^A \in MDP_l(I)$, if for any $a$ satisfying (\ref{at}),
   \[
   \liminf_{t\rightarrow\infty} \frac {t}{a^2(t)} \log \mathbb{P}(l_t^A\in G)\geq -\inf_G I,\quad G\subset \mathbb{R} \,\,\, \textrm{is}\,\,\, \textrm{open}.
   \]
\item[\bf{(3)}] $L_t^A$ is said to satisfy the uniform MDP with a rate function $I$, denoted by $L_t^A \in MDP(I)$, if $L_t^A \in MDP_u(I)$ and $L_t^A \in MDP_l(I)$.
\end{enumerate}
\end{defn}

The MDP has been established in \cite{Gao17} for non-degenerate SDEs by using Wang's Harnack inequality \cite{Wfy11}:
\[
\d X_t= b(X_t)\d t+ \sigma(X_t)\d B_t,\,\,\,X_0=x \in \mathbb{R}^d.
\]
The assumptions in \cite{Gao17} was further simplified and improved in \cite{WZ19}, so that degenerate situations are also included.

In this paper, we investigate MDP for unbounded additive functionals of the following distribution dependent SDE (DDSDE for short) on $\mathbb{R}^d$:
\beg{equation}\label{mDDSDE}
\beg{split}
&\d X_t= b(X_t,\mathscr{L}_{X_t})\d t+ \sigma(X_t,\mathscr{L}_{X_t})\d B_t,
\end{split}
\end{equation}
where $b:\mathbb{R}^d\times\mathscr{P}_2(\mathbb{R}^d)\rightarrow\mathbb{R}^d$,
$\sigma:\mathbb{R}^d\times\mathscr{P}_2(\mathbb{R}^d)\rightarrow \mathbb{R}^{d\times d}$,
$B_t$ is a $d$-dimensional Brownian motion, $\mathscr{L}_{X_t}$ is the law of $X_t$ under the reference probability space.

Let $\mathscr{P}_2$ be the space of all probability measures $\mu$ on $\mathbb{R}^d$ such that
\[
\|\mu\|_2:=\bigg( \int_{\mathbb{R}^d } |x|^2 \mu(\d x)  \bigg)^{\frac 1 2}< \infty.
\]
It is well known that $\mathscr{P}_2$ is a Polish space under the Wasserstein distance
\[
W_2(\mu,\nu):=\inf_{\pi \in \C(\mu,\nu)} \Bigg( \int_{\mathbb{R}^d \times \mathbb{R}^d} |x-y|^2 \pi(\d x,\d y) \Bigg)^{\frac 1 2},
\]
where $\C(\mu,\nu)$ is the set of all couplings for $\mu$ and $\nu$.

As in \cite{RW20}, to establish MDP for DDSDE (\ref{mDDSDE}), we choose a reference SDE whose solution is Markovian so that existing results on the MDP apply. By comparing the original equation with the reference one in the sense of  MDP, we establish the MDP for the DDSDE. We will state the main results in Section 2, and present complete proofs in Section 3.

\section{Main results}

We consider several different situations.

\subsection{Lipschitz Continuous $A$.}

We consider DDSDE (\ref{mDDSDE}) and make the following assumptions:

\begin{enumerate}
\item[\bf{(H1)}] $b$ is continuous, and $\sigma$ is Lipschitz continuous on $\mathbb{R}^d \times \mathscr{P}_2(\mathbb{R}^d)$ such that
\begin{align*}
&2\langle b(x,\mu)-b(y,\nu),x-y \rangle+ \| \sigma(x,\mu)-\sigma(y,\nu)\|_{HS}^2 \\
&\leq \lambda_2 W_2(\mu,\nu)^2- \lambda_1 |x-y|^2,\,\,\quad \,\,x,y\in\mathbb{R}^d;\,\,\,\mu,\nu\in \mathscr{P}_2(\mathbb{R}^d)
\end{align*}
holds for some constants $\lambda_1 > \lambda_2 \geq 0$.
\item[\bf{(H2)}] There exist constants $0<\kappa_1\leq \kappa_2<\infty$ such that
   \[
   \kappa_1^2 I \leq \sigma(x,\mu)\sigma(x,\mu)^{*}\leq \kappa_2^2 I,\quad x\in \mathbb{R}^d,~\mu\in \mathscr{P}_2(\mathbb{R}^d),
   \]
    where $\sigma^{*}$ denotes the transpose of the matrix $\sigma$, $I$ denotes the identity matrix.
\end{enumerate}

According to \cite[Theorem 2.1]{Wang18}, assumption (H1) implies that for any $X_0\in L^2(\Omega\rightarrow\mathbb{R}^d,\mathcal{F}_0,\mathbb{P})$, the equation (\ref{mDDSDE}) has a unique solution. We write $P_t^{*}\nu=\mathscr{L}_{X_t}$ if $\mathscr{L}_{X_0}=\nu$. By \cite[ Theorem 3.1(2)]{Wang18}, $P_t^{*}$ has a unique invariant probability measure $\bar{\mu}\in \mathscr{P}_2(\mathbb{R}^d)$
such that
\beg{equation}\label{W2}
W_2(P_t^{*}\nu,\bar{\mu})^2\leq W_2(\nu,\bar{\mu})^2 e^{-(\lambda_1-\lambda_2)t},\,\,\,~t\geq 0,~\,\nu\in \mathscr{P}_2(\mathbb{R}^d).
\end{equation}

Consider the stationary reference SDE:
\beg{equation}\label{refmSDE}
\beg{split}
&\d \bar{X}_t= b(\bar{X}_t,\bar{\mu})\d t+ \sigma(\bar{X}_t,\bar{\mu})\d B_t,\quad \mathscr{L}_{X_0}=\bar{\mu}.
\end{split}
\end{equation}
Under (H1), the equation (\ref{refmSDE}) has a unique solution $\bar{X}^x_t$ for any starting point $x \in \mathbb{R}^d$, and $\bar{\mu}$ is the unique invariant probability measure of the associated Markov semigroup
\[
\bar{P}_t f(x):=\mathbb{E}[f(\bar{X}^x_t)],\,\,\,t\geq 0,\,\,x \in \mathbb{R}^d,\,\,\,f\in \B_b(\mathbb{R}^d),
\]
where $\bar{P}_t$ is generated by
\[
\bar{\mathscr{A}}:=\frac 1 2\sum_{i,j=1}^d \{\sigma\sigma^{*}\}_{ij}(x,\bar{\mu})\partial_i\partial_j+\sum_{i=1}^d b_i(x,\bar{\mu})\partial_i.
\]

According to \cite{Gao17} and \cite{WZ19}, under assumptions (H1) and (H2),  $\bar{P}_t$ is $\bar{\mu}$-hypercontractive and strong Feller, i.e., $\|\bar{P}_t\|_{L^2(\bar{\mu})\rightarrow L^4(\bar{\mu})}=1$ for large $t> 0$ and $\bar{P}_t \B_b(\mathbb{R}^d)\subset C_b(\mathbb{R}^d)$ for $t> 0$. In particular, the hypercontractivity implies that there exists $\lambda> 0$ such that
\[
\bar{\mu}(|\bar{P}_t f-\bar{\mu}(f)|^2)\leq e^{-\lambda t} \bar{\mu}(|f-\bar{\mu}(f)|^2),\quad t\geq0,\quad f\in L^2(\bar{\mu}),
\]
so, for any $f\in L^2(\bar{\mu})$,
\beg{equation}\label{VarA}
\bar{V}(f):=\int_0^{\infty} \bar{\mu}(|\bar{P}_t f-\bar{\mu}(f)|^2) \d t <\infty.
\end{equation}

We have the following result:
\begin{thm}\label{mresult}
Assume (H1) and (H2). If $\mathbb{E}[e^{\delta |X_0|^2}]<\infty$ for some constant $\delta>0$, then for any Lipschitz continuous function $A$ on $\mathbb{R}^d$, $L_t^A\in \textrm{MDP}(I)$ for $I(y)={y^2}/({8 \bar{V}(A)})$, $y\in \mathbb{R}$.
\end{thm}

\subsection{H\"{o}lder continuous $A$.}

When $A$ is H\"{o}lder continuous, we need to assume that $\sigma(x,\mu)=\sigma(\mu)$ does not depend on $x$. In this case, the DDSDE becomes
\beg{equation}\label{maDDSDE}
\beg{split}
&\d X_t= b(X_t,\mathscr{L}_{X_t})\d t+ \sigma(\mathscr{L}_{X_t})\d B_t,
\end{split}
\end{equation}
and the reference SDE reduces to
\beg{equation}\label{refmaSDE}
\beg{split}
&\d \bar{X}_t= b(\bar{X}_t,\bar{\mu})\d t+ \sigma(\bar{\mu})\d B_t.
\end{split}
\end{equation}

Below we give the main result of this subsection.

\begin{thm}\label{mresult2}
Assume (H1),(H2), and let $\sigma(x,\mu)=\sigma(\mu)$ do not depend on $x$. If there exists a constant $\delta>0$ such that $\mathbb{E}[e^{\delta |X_0|^2}]<\infty$, then for any function $A$ such that
\[
\sup_{x\neq y}{\frac{|A(x)-A(y)|}{ |x-y|^{\alpha}\big( 1+|x|+|y| \big)^{2-\alpha}}}< \infty,\quad x,y\in \mathbb{R}^d
\]
holds for some $\alpha\in(0,1)$, $L_t^A\in \textrm{MDP}(I)$ for $I(y)={y^2}/({8 \bar{V}(A)})$, $y\in \mathbb{R}$.
\end{thm}

\subsection{Non-H\"{o}lder continuous $A$.}
In this part, we consider non-H\"{o}lder continuous $A$ for which we need to further strengthen the assumption that $\sigma$ is constant matrix. So, the DDSDE and the reference SDE reduce to
\beg{equation}\label{macDDSDE}
\beg{split}
&\d X_t= b(X_t,\mathscr{L}_{X_t})\d t+ \sigma \d B_t,
\end{split}
\end{equation}
and
\beg{equation}\label{refmacSDE}
\beg{split}
&\d \bar{X}_t= b(\bar{X}_t,\bar{\mu})\d t+ \sigma \d B_t.
\end{split}
\end{equation}

\begin{thm}\label{mresult3}
Assume (H1),(H2) and let $\sigma$ be constant. If $\mathbb{E}[e^{\delta |X_0|^2}]<\infty$ for some $\delta>0$, then for any function $A$ such that
\[
\sup_{x\neq y}{\frac{|A(x)-A(y)|\cdot{\log(e+|x|^2+|y|^2)\cdot[\log(e+|x-y|^{-1})]^{p}
}}{ \big( 1+|x|^2+|y|^2 \big)}}< \infty,\quad x,y\in \mathbb{R}^d
\]
holds for some $p>1$, $L_t^A\in \textrm{MDP}(I)$ for $I(y)={y^2}/({8 \bar{V}(A)})$, $y\in \mathbb{R}$.
\end{thm}

\subsection{The degenerate case}

In this section, we consider the distribution dependent stochastic Hamiltonian system for $X_t=(X_t^{(1)}, X_t^{(2)})$ on $\mathbb{R}^{m+d}$:
\begin{align}\label{SHS}
\left\{
\begin{aligned}
\d X_t^{(1)} & =  (A X_t^{(1)}+B X_t^{(2)}) \d t, \\
\d X_t^{(2)} & =  Z( X_t, \mathscr{L}_{X_t}) \d t+M \d B_t,
\end{aligned}
\right.
\end{align}
where $A,B$ and $M$ are $m\times m$, $m\times d$ and $d\times d$ matrixes respectively, $B_t$ is $d$ dimensional Brownian motion. Define
\[
W_2(\nu_1,\nu_2):=\inf_{\pi \in \C(\nu_1,\nu_2)} \Bigg( \int_{\mathbb{R}^{m+d} \times \mathbb{R}^{m+d}} \big(|\xi_1^{(1)}-\xi_2^{(1)}|^2+|\xi_1^{(2)}-\xi_2^{(2)}|^2\big) \pi\big(\d\xi_1,\d\xi_2\big) \Bigg)^{\frac 1 2}.
\]

We assume
\begin{enumerate}
\item[\bf{(D1)}] $M$ is invertible and $Rank[B,AB,\ldots,A^{m-1}B]=m$.

\item[\bf{(D2)}]  $Z: \mathbb{R}^{m+d} \times\mathscr{P}_2(\mathbb{R}^{m+d})\rightarrow \mathbb{R}^d$ is Lipschitz continuous.

\item[\bf{(D3)}] There exist constants $r>0$, $\theta_1>\theta_2>0$ and $r_0\in(-\|B\|^{-1},\|B\|^{-1})$ such that
\begin{align*}
&\langle r^2(x^{(1)}-y^{(1)})+r r_0 B(x^{(2)}-y^{(2)}), A(x^{(1)}-y^{(1)})+B(x^{(2)}-y^{(2)}) \rangle\\
&+\langle Z(x,\mu)-Z(y,\nu), x^{(2)}-y^{(2)}+ r r_0 B^{*}(x^{(1)}-y^{(1)}) \rangle\\
&\leq -\theta_1(|x^{(1)}-y^{(1)}|^2+|x^{(2)}-y^{(2)}|^2)+\theta_2 W_2(\mu,\nu)^2,\\
&\,\quad\,x=(x^{(1)},x^{(2)}),~y=(y^{(1)},y^{(2)})\in \mathbb{R}^{m+d},~\mu,\nu\in \mathscr{P}_2(\mathbb{R}^{m+d}).
\end{align*}
\end{enumerate}

\begin{thm}\label{mreSHS}
Assume (D1)-(D3), and let $A$ be a Lipschitz continuous function on $\mathbb{R}^{m+d}$. Then
\begin{enumerate}
\item[\bf{(1)}]
 For any $\mu_0,\nu_0\in \mathscr{P}_2(\mathbb{R}^{m+d})$, there exists a constant $C$ such that
 \[
 W_2(P_t^{*}\mu_0,P_t^{*}\nu_0)^2\leq C e^{-\frac{\theta_1-\theta_2}{2 C} t} W_2(\mu_0,\nu_0)^2,\quad\quad t\geq 0.
 \]
\item[\bf{(2)}]
$P_t^{*}$ has an invariant probability measure $\bar{\mu}\in \mathscr{P}_2(\mathbb{R}^{m+d})$ such that
\begin{align}\label{W2SHS}
W_2(P_t^{*}\mu_0,\bar{\mu})^2\leq C e^{-\frac{\theta_1-\theta_2}{2 C} t} W_2(\mu_0,\bar{\mu})^2 ,\quad\quad t\geq 0,~\mu_0\in \mathscr{P}_2(\mathbb{R}^{m+d}).
\end{align}
\item[\bf{(3)}]
If there exists a constant $\delta>0$ such that $\mathbb{E}[e^{\delta |X_0|^2}]<\infty$, then $L_t^A\in \textrm{MDP}(I)$ for $I(y)={y^2}/({8 \bar{V}(A)})$, $y\in \mathbb{R}^{m+d}$.

\end{enumerate}
\end{thm}

\section{Proofs of main results}

\subsection{Proof of Theorem \ref{mresult}}

To prove the Theorem \ref{mresult}, we will compare $l_t^A$ with the additive functional for $\bar{X}_t$. Let
\[
\bar{L}_t^A:= \frac 1 t \int_0^t A(\bar{X}_s)\d s,
\]
and
\[
\bar{l}_t^A:=\frac{t}{a(t)}(\bar{L}_t^A-\bar{\mu}(A))=\frac 1 {a(t)} \int_0^t \big(A(\bar{X}_s)-\bar{\mu}(A)\big) \d s,
\]
where $a(t)$ is a positive function satisfying (\ref{at}).

Define the Cram\'{e}r functional of $\bar{l}_t^A$:
\begin{align}\label{Lambdaz}
\Lambda(z)&:=\lim_{t\rightarrow +\infty} \frac{t}{a^2(t)} \log \mathbb{E}_x\bigg[ \exp\Big\{ \frac{a^2(t)}{t} z \bar{l}_t^A \Big\} \bigg]\\
&=\lim_{t\rightarrow +\infty} \frac{t}{a^2(t)} \log \mathbb{E}_x\bigg[ \exp\Big\{ \frac{a(t)}{t} z \int_0^t \big(A(\bar{X}_s)-\bar{\mu}(A)\big)\d s \Big\} \bigg]\nonumber,
\end{align}
where $\mathbb{E}_x$ is the expectation conditioned to $Y_0 = x$, $z$ is a constant. The Legendre transformation of $\Lambda(z)$ is defined by
\[
\Lambda^{*}(y):=\sup_{z\in \mathbb{R}^d}\{ zy-\Lambda(z)\},
\]
which is related to the rate function. According to the G\"{a}rtner-Ellis Theorem and \cite[Theorem 1.3]{Gao17}, $\bar{L}_t^A\in \textrm{MDP}(I)$ for $I(y)={y^2}/({8 \bar{V}(A)})$, $y\in \mathbb{R}$.

Below we introduce the following exponential approximation lemma which is useful in applications, see for instance \cite[Theorem 4.2.16]{DZ98} and \cite[Theorem 3.2]{RWW06}.

\begin{lem}\label{Exappro}(Exponential approximations)
If $\bar{L}_t^A \in \textrm{MDP}_u(I)$(respectively $\textrm{MDP}_l(I)$) and for any $a$ satisfying (\ref{at}),
\[
\lim_{t\rightarrow\infty} \frac {t}{a^2(t)} \log \mathbb{P}( |l_t^A -\bar{l}_t^A|> \varepsilon )=-\infty,\,\,\forall\varepsilon> 0,
\]
then $L_t^A \in \textrm{MDP}_u(I)$(respectively $\textrm{MDP}_l(I)$).
\end{lem}

Using this Lemma, we prove the following result, which is crucial in the present study.

\begin{thm}\label{ThEeN}
If $\bar{L}_t^A \in \textrm{MDP}_u(I)$(respectively $\textrm{MDP}_l(I)$) and there exists a constant $\delta>0$ such that
\begin{align}\label{EeN}
\mathbb{E}\Big[\exp\Big\{ \delta \int_0^{\infty} |X_s -\bar{X}_s| \d s\Big\} \Big]<\infty,
\end{align}
then $L_t^A \in \textrm{MDP}_u(I)$(respectively $\textrm{MDP}_l(I)$).
\end{thm}

\begin{proof}
Due to that $\lim_{t\rightarrow\infty} \frac{a(t)}{t}=0$, there exists $t_0> 0$ such that $\sqrt{\frac{a(t)}{t}}\leq \delta$ when $t\geq t_0$. Below, we assume that $t\geq t_0$ and we have
\[
\mathbb{P}( |l_t^A -\bar{l}_t^A|> \varepsilon )\leq \mathbb{P}\Big( \sqrt{\frac{a(t)}{t}}\int_0^t |X_s- \bar{X}_s|\d s > \sqrt{\frac{a(t)}{t}} \frac{a(t)\varepsilon}{K} \Big),
\]
by Chebyshev's inequality, we obtain
\begin{align*}
\mathbb{P}( |l_t^A -\bar{l}_t^A|> \varepsilon )
&\leq \frac{\mathbb{E}\Big[\exp\Big\{ \sqrt{\frac{a(t)}{t}}\int_0^t |X_s- \bar{X}_s|\d s \Big\} \Big]}{\exp\{ \sqrt{\frac{a(t)}{t}} \frac{a(t)\varepsilon}{K}\}},
\end{align*}
then (\ref{at}) and (\ref{EeN}) imply that $\forall\varepsilon> 0$,
\begin{align*}
&\lim_{t\rightarrow\infty} \frac {t}{a^2(t)} \log \mathbb{P}( |l_t^A -\bar{l}_t^A|> \varepsilon )\\
&\leq \lim_{t\rightarrow\infty}\frac {t}{a^2(t)} \Big(\log\mathbb{E}\Big[\exp\Big\{ \sqrt{\frac{a(t)}{t}}\int_0^t |X_s- \bar{X}_s|\d s \Big\} \Big]-\sqrt{\frac{a(t)}{t}}\frac{a(t)\varepsilon}{K} \Big)\\
&\leq \lim_{t\rightarrow\infty} \frac {t}{a^2(t)} \log\mathbb{E}\Big[\exp\Big\{ \delta\int_0^t |X_s- \bar{X}_s|\d s \Big\} \Big]- \lim_{t\rightarrow\infty} \sqrt{\frac{t}{a(t)}}\frac{\varepsilon}{K}\\
&=-\infty.
\end{align*}
Then the desired assertion follows from Lemma \ref{Exappro}.
\end{proof}

\begin{proof}[Proof of Theorem \ref{mresult}]
Let $\mathscr{L}_{X_0}=\nu$ and $\mathscr{L}_{\bar{X}_0}=\bar{\mu}$. According to \cite[Theorem 1.1-1.3]{Gao17}, $\bar{L}_t^A\in \textrm{MDP}(I)$. So, it suffices to show (\ref{EeN}) for some $\delta> 0$.

Condition (H1) implies that the reference SDE (\ref{refmSDE}) is well-posed and the solution is a Markov process, $\bar{\mu}$ is the unique invariant probability measure of $P_t^{*}$. Simply denote $X_t=X_t^{\nu}, \bar{X}_t=\bar{X}_t^x$ and $P_t^{*}\nu=\mathscr{L}_{X_t^{\nu}}$ for $\nu\in \mathscr{P}_2(\mathbb{R}^d)$. By It\^{o}'s formula and (H1),
\begin{align*}
\d |X_t- \bar{X}_t|^2 &\leq \big\{ \lambda_2 W_2(P_t^{*}\nu,\bar{\mu})^2-\lambda_1 |X_t- \bar{X}_t|^2  \big\}\d t \\
&\quad + 2\big\langle X_t- \bar{X}_t, \big( \sigma(X_t,P_t^{*}\nu)-\sigma(\bar{X}_t,\bar{\mu}) \big)\d B_t \big\rangle.
\end{align*}
Let $\xi_t=\big( e^{-\lambda t}+|X_t- \bar{X}_t|^2 \big)^{\frac 1 2}$, where $\lambda:=\lambda_1-\lambda_2$. By (\ref{W2}), we find a constant $C>0$ such that
\begin{align*}
\d \xi_t \leq -\frac{\lambda_1}{2} \xi_t \d t+C e^{-\frac{\lambda}{2} t} \d t + \d M_t,
\end{align*}
where $\d M_t= \frac{1}{\xi_t}\big\langle X_t- \bar{X}_t, \big( \sigma(X_t,P_t^{*}\nu)-\sigma(\bar{X}_t,\bar{\mu}) \big)\d B_t \big\rangle$. Therefore, for some $\delta>0$, we obtain that
\begin{align*}
\mathbb{E} \big[e^{\delta \int_0^t \xi_s \d s}\big]
&\leq e^{\frac{4\delta C}{\lambda_1\lambda}}  \mathbb{E} \Big[ e^{\frac{2\delta \xi_0}{\lambda_1}} e^{\frac{2\delta }{\lambda_1} \int_0^t \d M_s}\Big]\\
&=e^{\frac{4\delta C}{\lambda_1\lambda}}  \mathbb{E}\Big[ \mathbb{E} \big[ e^{\frac{2\delta \xi_0}{\lambda_1}} e^{\frac{2\delta }{\lambda_1} \int_0^t \d M_s} | \mathcal{F}_0\big]\Big]\\
&\leq  C(\delta) \Big(\mathbb{E}\Big[ e^{\frac{4\delta \xi_0}{\lambda_1}}\Big]\Big)^{\frac 1 2} \Big(\mathbb{E}\Big[ e^{ \frac{64\delta^3\kappa_2\sqrt{d}}{\lambda_1^2} \int_0^t|X_s- \bar{X}_s| \d s} \Big]\Big)^{\frac 1 4},\,\,\,t> 0
\end{align*}
holds for some constant $C(\delta)>0$. Therefore, we obtain that
$$\mathbb{E} e^{\delta \int_0^\infty |X_s- \bar{X}_s| \d s} \leq \mathbb{E} e^{\delta \int_0^\infty \xi_s \d s}<\infty$$ for $\delta>0$ small enough. Therefore, there exists some constant $\delta> 0$ such that (\ref{EeN}) holds.
\end{proof}

\subsection{Proof of Theorem \ref{mresult2}}

In order to prove Theorem \ref{mresult2}, we need the following result.

\begin{thm}\label{ThEeXY}
If $\bar{L}_t^A \in \textrm{MDP}_u(I)$(respectively $\textrm{MDP}_l(I)$) and there exists a constant $\delta>0$ such that
\begin{align}\label{EeN12}
\mathbb{E}\Big[\exp\Big\{ \delta \int_0^{\infty} |X_s -\bar{X}_s|^{\alpha}\big( 1+|X_s|+|\bar{X}_s| \big)^{2-\alpha} \d s\Big\} \Big]<\infty,
\end{align}
then $L_t^A \in \textrm{MDP}_u(I)$(respectively $\textrm{MDP}_l(I)$).
\end{thm}

The proof is similar to that of Theorem \ref{ThEeN}, so we omit to save space.

\begin{proof}[Proof of Theorem \ref{mresult2}]
Let $\mathscr{L}_{X_0}=\nu$ and $\mathscr{L}_{\bar{X}_0}=\bar{\mu}$. According to \cite[Theorem 2.1]{WZ19}, $\bar{L}_t^A\in \textrm{MDP}(I)$ for $I(y)={y^2}/({8 \bar{V}(A)})$, $y\in \mathbb{R}$. Therefore, it suffices to show (\ref{EeN12}) for some $\delta> 0$.

The assumption (H1) and (\ref{W2}) yield
\[
\|\sigma(P_t^{*}\nu)-\sigma(\bar{\mu})\|_{HS}^2\leq \lambda_2 W_2(P_t^{*}\nu,\bar{\mu})^2\leq \lambda_2 e^{-(\lambda_2-\lambda_1)t}W_2(\nu,\bar{\mu})^2.
\]
By Young's inequality and $C_r$ inequality, for any $\lambda>0$,
\begin{align*}
|X_t -\bar{X}_t|^{\alpha}\big( 1+|X_t|+|\bar{X}_t| \big)^{2-\alpha}&=|X_t -\bar{X}_t|^{\alpha} e^{\lambda t} e^{-\lambda t}\big( 1+|X_t|+|\bar{X}_t| \big)^{2-\alpha}\\
&\leq \frac {\alpha} 2 e^{\frac {2\lambda t}{\alpha}} |X_t- \bar{X}_t|^2+ \frac {2-\alpha} 2 e^{-\frac {2\lambda t}{2-\alpha}} 3 \big( 1+|X_t|^2+|\bar{X}_t|^2 \big).
\end{align*}
Below, for simplicity, we take $\lambda<\frac{\alpha(\lambda_1-\lambda_2)}{2}$. By H\"{o}lder's inequality, we have
\begin{align*}
&\mathbb{E}\Big[e^{ \delta \int_0^{\infty} |X_s -\bar{X}_s|^{\alpha}( 1+|X_s|+|\bar{X}_s| )^{2-\alpha} \d s} \Big]\\
&\leq \Big(\mathbb{E}\Big[e^{ \alpha\delta \int_0^{\infty} e^{\frac{2\lambda s}{\alpha}}|X_s -\bar{X}_s|^2 \d s} \Big]\Big)^{\frac 1 2}  \Big(\mathbb{E}\Big[e^{ 3(2-\alpha)\delta\int_0^{\infty} e^{-\frac{2\lambda s}{2-\alpha}}( 1+|X_s|^2+|\bar{X}_s|^2 ) \d s} \Big]\Big)^{\frac 1 2}.
\end{align*}
Below search for a constant $\delta> 0$ such that
\[
I_1:=\mathbb{E}\Big[e^{ \alpha\delta \int_0^{\infty} e^{\frac{2\lambda s}{\alpha}}|X_s -\bar{X}_s|^2 \d s} \Big]<\infty,
\]
and
\[
I_2:=\mathbb{E}\Big[e^{ 3(2-\alpha)\delta\int_0^{\infty} e^{-\frac{2\lambda s}{2-\alpha}}( 1+|X_s|^2+|\bar{X}_s|^2 ) \d s} \Big]<\infty.
\]

By It\^{o}'s formula and (H1), we have
\[
\d |X_t- \bar{X}_t|^2\leq \big( \lambda_2 W_2(P_t^{*}\nu,\bar{\mu})^2- \lambda_1 |X_t- \bar{X}_t|^2 \big)\d t+2\langle X_t- \bar{X}_t,\big(\sigma(P_t^{*}\nu)-\sigma(\bar{\mu})\big)\d B_t \rangle.
\]
By the chain rule, we obtain
\begin{align*}
&\d \{e^{\frac{2\lambda t}{\alpha}}|X_t -\bar{X}_t|^2\}\\
&\leq e^{\frac{2\lambda t}{\alpha}}\Big\{\frac{2\lambda }{\alpha}|X_t -\bar{X}_t|^2+\lambda_2 W_2(P_t^{*}\nu,\bar{\mu})^2- \lambda_1 |X_t -\bar{X}_t|^2 \Big\}\d t\\
&\,\,\,\,\,+ 2 e^{\frac{2\lambda t}{\alpha}}\langle X_t -\bar{X}_t,\big(\sigma(P_t^{*}\nu)-\sigma(\bar{\mu})\big)\d B_t \rangle.
\end{align*}
Therefore, we obtain that
\begin{align*}
&\mathbb{E}\Big[e^{ \alpha\delta \int_0^t e^{\frac{2\lambda s}{\alpha}}|X_s -\bar{X}_s|^2 \d s} \Big]\\
&\leq C_1(\delta)\mathbb{E}\Big[ e^{\frac{\alpha \delta |X_0 -\bar{X}_0|^2}{\lambda_1-2\lambda/{\alpha}}+\frac{2\alpha\delta}{\lambda_1-2\lambda/{\alpha}}
\int_0^t e^{\frac{2\lambda s}{\alpha}}\langle X_s -\bar{X}_s,(\sigma(P_s^{*}\nu)-\sigma(\bar{\mu}))\d B_s \rangle} \Big]\\
&= C_1(\delta)\mathbb{E}\Big[ \mathbb{E}\big[ e^{\frac{\alpha \delta |X_0 -\bar{X}_0|^2}{\lambda_1-2\lambda/{\alpha}}+\frac{2\alpha\delta}{\lambda_1-2\lambda/{\alpha}}
\int_0^t e^{\frac{2\lambda s}{\alpha}}\langle X_s -\bar{X}_s,(\sigma(P_s^{*}\nu)-\sigma(\bar{\mu}))\d B_s \rangle} | \mathcal{F}_0\big]\Big]\\
&\leq C_1(\delta)\Big( \mathbb{E}\Big[e^{\frac{\alpha \delta |X_0 -\bar{X}_0|^2}{\lambda_1-2\lambda/{\alpha}}}\Big]\Big)^{\frac 1 2}\Big( \mathbb{E}\Big[ e^{\frac{32 \alpha^2\delta^2 \lambda_2 W_2(\nu,\bar{\mu})^2}{(\lambda_1-2\lambda/{\alpha})^2}\int_0^t e^{\frac{2\lambda s}{\alpha}}|X_s -\bar{X}_s|^2 \d s } \Big] \Big)^{\frac 1 4},
\end{align*}
where $C_1(\delta):=e^{\frac{\alpha\delta \lambda_2 W_2(\nu,\bar{\mu})^2 }{(\lambda_1-2\lambda/{\alpha})(\lambda_1-\lambda_2-2\lambda/{\alpha})}}$. We choose $\delta$ such that $\delta\leq \frac{(\lambda_1-2\lambda/{\alpha})^2}{32 \alpha\lambda_2 W_2(\nu,\bar{\mu})^2} $, which leads to $I_1<\infty$.

Next we need to prove that $I_2<\infty$. By (B1), there exist constants $c_1,c_2> 0$ such that
\begin{align}\label{dXt1}
\d |X_t|^2\leq \big( c_2- c_1 |X_t|^2 +c_2 W_2(P_t^{*}\nu,\bar{\mu})^2 \big)\d t+2\langle X_t, \sigma(P_t^{*}\nu)\d B_t\rangle,
\end{align}
\begin{align}\label{dXt2}
\d |\bar{X}_t|^2\leq \big( c_2-c_1|\bar{X}_t|^2 + c_2 \|\bar{\mu}\|_2^2 \big)\d t+2\langle \bar{X}_t, \sigma(\bar{\mu})\d B_t\rangle,
\end{align}
and
\[
\|\sigma(P_t^{*}\nu)\|_{HS}^2\leq c_2 \big( 1+ W_2(P_t^{*}\nu,\bar{\mu})^2 \big).
\]
Recall that
\begin{align*}
I_2&:=\mathbb{E}\Big[e^{ 3(2-\alpha)\delta\int_0^{\infty} e^{-\frac{2\lambda s}{2-\alpha}}( 1+|X_s|^2+|\bar{X}_s|^2 ) \d s} \Big]\\
&= \mathbb{E}\Big[e^{ 3(2-\alpha)\delta \int_0^{\infty} e^{-\frac{2\lambda s}{2-\alpha}}\d s+3(2-\alpha)\delta \int_0^{\infty} e^{-\frac{2\lambda s}{2-\alpha}} |X_s|^2 \d s+3(2-\alpha)\delta \int_0^{\infty} e^{-\frac{2\lambda s}{2-\alpha}}|\bar{X}_s|^2 \d s} \Big]\\
&\leq e^{\frac{3\delta(2-\alpha)^2}{2\lambda}}\Big( \mathbb{E}\Big[e^{6 \delta{(2-\alpha)} \int_0^{\infty} e^{-\frac{2\lambda s}{2-\alpha}}|X_s|^2 \d s} \Big] \Big)^{\frac 1 2}\Big( \mathbb{E}\Big[e^{6 \delta{(2-\alpha)} \int_0^{\infty} e^{-\frac{2\lambda s}{2-\alpha}}|\bar{X}_s|^2 \d s} \Big] \Big)^{\frac 1 2}.
\end{align*}
Thus, for $I_2< \infty$, it suffices to show
\[
I_2':=\mathbb{E}\Big[e^{6 \delta{(2-\alpha)} \int_0^{\infty} e^{-\frac{2\lambda s}{2-\alpha}}|X_s|^2 \d s} \Big]<\infty,
\]
and
\[
I_2'':=\mathbb{E}\Big[e^{6 \delta{(2-\alpha)} \int_0^{\infty} e^{-\frac{2\lambda s}{2-\alpha}}|\bar{X}_s|^2 \d s} \Big]<\infty.
\]

By the chain rule and (\ref{dXt1}), we have
\begin{align*}
&\d \{e^{-\frac{2\lambda t}{2-\alpha}}|X_t|^2\}\\
&\leq e^{-\frac{2\lambda t}{2-\alpha}}\big\{\big(-\frac{2\lambda}{2-\alpha}-c_1\big)|X_t|^2 +c_2 + c_2 W_2(P_t^{*}\nu,\bar{\mu})^2 \big\}\d t+ 2 e^{-\frac{2\lambda t}{2-\alpha}}\langle X_t,\sigma(P_t^{*}\nu)\d B_t \rangle.
\end{align*}
Then
\begin{align*}
&\mathbb{E}\Big[e^{6 \delta{(2-\alpha)} \int_0^{t} e^{-\frac{2\lambda s}{2-\alpha}}|X_s|^2 \d s} \Big]\\
&\leq C_2(\delta) \mathbb{E}\Big[ e^{\frac{6 \delta{(2-\alpha)}|X_0|^2}{2\lambda/(2-\alpha)+c_1}}e^{\frac{12 \delta{(2-\alpha)}}{2\lambda/(2-\alpha)+c_1}\int_0^t e^{-\frac{2\lambda s}{2-\alpha}}\langle X_s,\sigma(P_s^{*}\nu)\d B_s \rangle} \Big]\\
&\leq C_2(\delta)\Big( \mathbb{E}\Big[e^{\frac{12 \delta{(2-\alpha)}|X_0|^2}{2\lambda/(2-\alpha)+c_1}}\Big]\Big)^{\frac 1 2}\Big( \mathbb{E}\Big[ e^{\frac{288 \delta^2 (2-\alpha)^4 c_2(1+W_2(\nu,\bar{\mu})^2)}{(2\lambda+c_1(2-\alpha))^2}\int_0^t e^{-\frac{2\lambda s}{2-\alpha}}|X_s|^2 \d s } \Big] \Big)^{\frac 1 4},
\end{align*}
where $C_2(\delta)=\exp\big\{6 c_2 \delta(2-\alpha)^3\big(\frac{1}{4\lambda^2+2\lambda c_1(2-\alpha)}+\frac{1}{(2\lambda +c_1(2-\alpha))(2\lambda+(\lambda_1-\lambda_2)(2-\alpha))}\big)\big\}$. Then for $\delta\leq\frac{(2\lambda+c_1(2-\alpha))^2}{48 \delta (2-\alpha)^3 c_2(1+W_2(\nu,\bar{\mu})^2)}$, we have $I_2'<\infty$.

On the other hand, the same argument gives
\begin{align*}
\mathbb{E}\Big[e^{6 \delta{(2-\alpha)} \int_0^{t} e^{-\frac{2\lambda s}{2-\alpha}}|\bar{X}_s|^2 \d s} \Big]\leq C_3(\delta)\Big( \mathbb{E}\Big[e^{\frac{12 \delta{(2-\alpha)}|\bar{X}_0|^2}{2\lambda/(2-\alpha)+c_1}}\Big]\Big)^{\frac 1 2} \Big(\mathbb{E}\Big[ e^{\frac{288  \delta^2(2-\alpha)^4 \|\sigma(\bar{\mu})\|_{HS}^2}{(2\lambda+c_1(2-\alpha))^2}\int_0^t e^{-\frac{2\lambda s}{2-\alpha}} |\bar{X}_s|^2 \d s} \Big]\Big)^{\frac 1 4}
\end{align*}
where $C_3(\delta)=\exp\{\frac{6\delta(2-\alpha)^3 c_2(1+ \|\bar{\mu}\|_2^2)}{4\lambda^2+2\lambda c_1(2-\alpha)}\}$. So, when $\delta\leq\frac{(2\lambda+c_1(2-\alpha))^2}{48 \delta (2-\alpha)^3 \|\sigma(\bar{\mu})\|_{HS}^2}$ we have $I_2''<\infty$.

Finally, we take $\delta\leq \min\Big\{ \frac{(\lambda_1-2\lambda/{\alpha})^2}{32 \alpha\lambda_2 W_2(\nu,\bar{\mu})^2}, \frac{(2\lambda+c_1(2-\alpha))^2}{48 \delta (2-\alpha)^3 c_2(1+W_2(\nu,\bar{\mu})^2)}, \frac{(2\lambda+c_1(2-\alpha))^2}{48 \delta (2-\alpha)^3 \|\sigma(\bar{\mu})\|_{HS}^2}\Big\}$. We conclude, there exists $\delta>0$ such that $I_2<\infty$, which together with $I_1<\infty$ finishes the proof.
\end{proof}

\subsection{Proof of Theorem \ref{mresult3}}

Let $\mathscr{L}_{X_0}=\nu$ and $\mathscr{L}_{\bar{X}_0}=\bar{\mu}$. According to \cite[Theorem 2.1]{WZ19}, $\bar{L}_t^A\in \textrm{MDP}(I)$ for $I(y)={y^2}/({8 \bar{V}(A)})$, $y\in \mathbb{R}$. Therefore, by the Lemma \ref{Exappro} (see also \cite[Theorem 4.2.16]{DZ98} or \cite[Theorem 3.2]{RWW06}), it suffices to prove
\begin{align}\label{EeN1}
\mathbb{E}\Big[\exp\Big\{ \delta \int_0^{\infty} \frac{\big( 1+|X_s|^2+|\bar{X}_s|^2 \big)}{\log(e+|X_s|^2+|\bar{X}_s|^2)[\log(e+|X_s-\bar{X}_s|^{-1})]^{p}
} \d s\Big\} \Big]<\infty
\end{align}
for some constant $\delta> 0$.

By the chain rule and (H1), we have
\begin{align*}
\d \big( e^{\lambda_1 t}|X_t- \bar{X}_t|^2\big)&= e^{\lambda_1 t}\big\{ \lambda_1|X_t-\bar{X}_t|^2 \d t+ \d |X_t- \bar{X}_t|^2 \big\}  \\
&\leq \lambda_2 e^{\lambda_1 t} W_2(P_t^{*}\nu,\bar{\mu})^2 \d t.
\end{align*}
Then we obtain
\begin{align*}
|X_t- \bar{X}_t|^2 \leq e^{-(\lambda_1-\lambda_2)t} W_2(\nu,\bar{\mu})^2,
\end{align*}
this implies that
\begin{align*}
|X_t- \bar{X}_t|^{-1} \geq e^{\frac{\lambda_1-\lambda_2}{2}t}W_2(\nu,\bar{\mu})^{-1}.
\end{align*}
Let $\alpha=\sup_{t\geq0}\big(|X_t|^2+|\bar{X}_t|^2 \big)$, $\beta=W_2(\nu,\bar{\mu})$, $\lambda=\frac{\lambda_1-\lambda_2}{2}$, then we have
\begin{align*}
\int_0^\infty |A(X_t)-A(\bar{X}_t)|\d t\leq \frac{e+\alpha}{\log(e+\alpha)} \int_0^\infty \frac{\d t}{[\log(e+\beta^{-1} e^{\lambda t})]^p}.
\end{align*}
Let $\beta^{-1}e^{\lambda t}=s$, then we have $\d t=\frac{\d s}{\lambda s} $, so that
\begin{align*}
\int_0^\infty \frac{\d t}{[\log(e+\beta^{-1} e^{\lambda t})]^p}
&=\frac{1}{\lambda} \int_{\beta^{-1}}^\infty \frac{\d s}{s[\log(e+s)]^p} \\
&\leq \frac{1}{\lambda} \int_{\beta^{-1}}^{\beta^{-1}+1}\frac{\d s}{s}+\frac{1}{\lambda}  \int_{\beta^{-1}+1}^\infty \big(1+\frac{e}{1+\beta^{-1}}\big)\big(\log(e+s)\big)^{-p} \d{\log(e+s)}\\
&\leq \frac{\log(1+\beta)}{\lambda}+\frac{1+e}{\lambda(p-1)}.
\end{align*}
Thus,
\begin{align}\label{J}
\int_0^\infty |A(X_t)-A(\bar{X}_t)|\d t\leq \frac{1}{\lambda}\Big\{\frac{(e+\alpha)(1+e)}{p-1} + J \Big\},
\end{align}
where $J:=\frac{e+\alpha}{\log(e+\alpha)}\cdot\log(e+\beta)$. Let
$$h(\alpha)=\frac{e+\alpha}{\log(e+\alpha)}\cdot\log(e+\beta)-\alpha.$$

When $\alpha\geq\beta$, we have $h'(\alpha)\leq 0$, which implies that $h$ decreases with respect to $\alpha$, and we obtain that $J\leq \alpha+e$.

When $0<\alpha<\beta$, let $g(\alpha)=\frac{e+\alpha}{\log(e+\alpha)}$, we have $g'(\alpha)\geq 0$, which implies that $g$ increases in $\alpha$, so that $\sup_{\alpha\in(0,\beta)} h(\alpha)\leq \frac{e+\beta}{\log(e+\beta)}\cdot\log(e+\beta)=e+\beta$.

Combining this with (\ref{J}), we find a constant $C_0> 0$ such that
\begin{align*}
\int_0^\infty |A(X_t)-A(\bar{X}_t)|\d t&\leq C_0(e+\alpha+\beta)\\
&= C_0 \Big\{\sup_{t>0}\{ |X_t|^2+|\bar{X}_t|^2 \} +W_2(\nu,\bar{\mu}) +e \Big\}.
\end{align*}
Since $\mathbb{E}[e^{\delta |X_0|^2}]+\bar{\mu}(e^{\delta |\cdot|^2})<\infty$ for some $\delta> 0$ and $\mathscr{L}_{\bar{X}_t}=\bar{\mu}$, (\ref{EeN1}) follows if
\begin{align}\label{EeXt}
\mathbb{E}\Big[ \sup_{t>0}  e^{ \delta |X_t|^2}  \Big]<\infty
\end{align}
holds for some $\delta> 0$.

Indeed, (\ref{EeXt}) holds also for $\bar{X}_t$ replacing $X_t$, since when $\mathscr{L}_{\bar{X}_0}=\bar{\mu}$, we have $\mathscr{L}_{(X_t)_{t\geq 0}}=\mathscr{L}_{(\bar{X}_t)_{t\geq 0}}$.

By (H1), there exists a constant $C_1>0$ such that
\begin{align*}
\d \left( e^{(\lambda_1-\lambda_2) t}|X_t|^2 \right)\leq C_1 e^{(\lambda_1-\lambda_2) t}\d t+ 2 e^{(\lambda_1-\lambda_2) t}\langle X_t, \sigma \d B_t\rangle.
\end{align*}
So,
\begin{align*}
\delta |X_t|^2 \leq \frac{\delta C_1}{\tilde{\lambda}} +\delta e^{-\tilde{\lambda} t}|X_0|^2 + 2 \delta e^{-\tilde{\lambda} t} \int_0^t e^{\tilde{\lambda} s} \langle X_s, \sigma \d B_s \rangle,
\end{align*}
where $\tilde{\lambda}:=\lambda_1-\lambda_2$. Therefore, we obtain
\begin{align}\label{EeXs}
\mathbb{E}\Big[ \sup_{0\leq s\leq t}  e^{ \delta |X_s|^2}  \Big]&\leq
e^{\delta C_1/{\tilde{\lambda}}} \mathbb{E}\Big[\mathbb{E}\Big[ \sup_{0\leq s\leq t}e^{ \delta |X_0|^2} \cdot e^{ 2 \delta e^{-\tilde{\lambda} s} \int_0^s e^{\tilde{\lambda} u} \langle X_u, \sigma \d B_u \rangle}| \mathcal{F}_0 \Big]\Big]\nonumber\\
&\leq e^{\delta C_1/{\tilde{\lambda}}} \Big(\mathbb{E}\Big[ e^{ 2 \delta |X_0|^2 } \Big]\Big)^{\frac 1 2}\cdot \Big(\mathbb{E}\Big[ \sup_{0\leq s\leq t} e^{ 4 \delta e^{-\tilde{\lambda} s} \int_0^s e^{\tilde{\lambda} u} \langle X_u, \sigma \d B_u \rangle} \Big]\Big)^{\frac 1 2}.
\end{align}
By the BDG inequality, there exists a constant $C_2> 0$, such that
\begin{align*}
\tilde{J}:=\mathbb{E}\Big[ \sup_{0\leq s\leq t} e^{ 4 \delta e^{-\tilde{\lambda} s} \int_0^s e^{\tilde{\lambda} u} \langle X_u, \sigma \d B_u \rangle} \Big]&\leq C_2 \Big(\mathbb{E}\Big[ e^{ 16 \delta^2 e^{-2\tilde{\lambda} t} \int_0^t e^{2\tilde{\lambda} u} |\sigma^{*} X_u|^2 \d u } \Big]\Big)^{\frac 1 2}\\
&=C_2 \Big(\mathbb{E}\Big[ e^{ 16 \delta^2 \int_0^t \|\sigma\|^2 |X_s|^2 \frac{1-e^{-2\tilde{\lambda} t}}{2\tilde{\lambda}} \frac{2\tilde{\lambda}}{1-e^{-2\tilde{\lambda} t}} e^{-2\tilde{\lambda} (t-s)} \d s } \Big]\Big)^{\frac 1 2}.
\end{align*}
Since $\lambda_t(\d s):=\frac{2\tilde{\lambda}}{1-e^{-2\tilde{\lambda} t}} e^{-2\tilde{\lambda} (t-s)} \d s$ is a probability measure on $[0,t]$, therefore by the Jensen's inequality, we obtain
\begin{align*}
\tilde{J}&\leq C_2 \Big(\mathbb{E}\Big[ e^{ \frac{16 \delta^2 (1-e^{-2\tilde{\lambda} t})}{2\tilde{\lambda}} \int_0^t \|\sigma\|^2 |X_s|^2 \lambda_t (\d s)} \Big]\Big)^{\frac 1 2}\\
&\leq C_2 \Big(\mathbb{E}\Big[ \int_0^t  e^{ \frac{8 \delta^2 \|\sigma\|^2 |X_s|^2 }{\tilde{\lambda}} }\lambda_t (\d s) \Big]\Big)^{\frac 1 2}.
\end{align*}
When $t\geq 1$, we have
\begin{align*}
\tilde{J}\leq \frac{C_2^2}{4\tilde{\lambda}}+\tilde{\lambda}\int_0^t \mathbb{E}\Big[ e^{  \delta |X_s|^2 -2\tilde{\lambda}(t-s)} \Big]\d s.
\end{align*}
Substituting into (\ref{EeXs}) and applying the Gronwall's lemma, we obtain
\begin{align}\label{EeXs2}
\mathbb{E}\Big[ \sup_{0\leq s\leq t} e^{ \delta  |X_s|^2 } \Big]&\leq C_3\Big( 1+ \mathbb{E}\Big[ \sup_{0\leq s\leq 1} e^{ \delta  |X_s|^2 } \Big]\Big)
\end{align}
for some constant $C_3> 0$. Finally,
\begin{align*}
\mathbb{E}\Big[ \sup_{0\leq t\leq 1} e^{ \delta  |X_t|^2 } \Big]\leq
\Big(\mathbb{E}\Big[ e^{2\delta |X_0|^2}\Big]\Big)^{\frac 1 2}\cdot e^{c\delta \int_0^1 (1+W_2(P_t^{*}\nu,\bar{\mu})^2)\d t}  \cdot \Big(\mathbb{E}\Big[ \sup_{0\leq t\leq 1} e^{4\delta \int_0^t \langle X_s,\sigma \d B_s\rangle }\Big]\Big)^{\frac 1 2}.
\end{align*}
By the exponential martingale inequality and the Jensen's inequality, we obtain that there exists a constant $C_\delta$ such that
\begin{align*}
\mathbb{E}\Big[ \sup_{0\leq t\leq 1} e^{ \delta  |X_t|^2 } \Big]&\leq
C_\delta \sqrt{e} \Big( \int_0^1 \mathbb{E}\big[ e^{16\delta^2 \|\sigma\|^2 |X_s|^2}\big]\d s \Big)^{\frac 1 4}.
\end{align*}
Taking $\delta\leq {1}/({ 16 \|\sigma\|^2})$ and we obtain that $\mathbb{E}\Big[ \sup_{0\leq t\leq 1} e^{ \delta  |X_t|^2 } \Big]<\infty$. This together with (\ref{EeXs2}) imply that $\mathbb{E}\Big[ \sup_{0\leq s\leq t} e^{ \delta  |X_s|^2 } \Big]<\infty$.

\subsection{Proof of Theorem \ref{mreSHS}}

Let
\[
\rho(\xi_1,\xi_2):=\left(|\xi_1^{(1)}-\xi_2^{(1)}|^2+|\xi_1^{(2)}-\xi_2^{(2)}|^2\right)^{1/2}.
\]
We take $X_0,Y_0\in L^2(\Omega\rightarrow\mathbb{R}^{m+d},\mathcal{F}_0,\mathbb{P})$ such that $\mathscr{L}_{X_0}=\mu_0,\mathscr{L}_{Y_0}=\nu_0$ and
\[
W_2(\mu_0,\nu_0)^2=\mathbb{E} \rho(X_0,Y_0)^2.
\]
Let $X_t=(X_t^{(1)},X_t^{(2)})$ and $Y_t=(Y_t^{(1)},Y_t^{(2)})$ solve (\ref{mreSHS}) with initial values $X_0$ and $Y_0$ respectively. Obviously, $X_t^{(1)}-Y_t^{(1)}$ and $X_t^{(2)}-Y_t^{(2)}$ solve the ODE
\begin{align}\label{ODE}
\left\{
\begin{aligned}
\d (X_t^{(1)}-Y_t^{(1)}) & =  \Big(A (X_t^{(1)}-Y_t^{(1)})+B (X_t^{(2)}-Y_t^{(2)})\Big) \d t, \\
\d (X_t^{(2)}-Y_t^{(2)}) & =  \Big(Z( X_t, \mathscr{L}_{X_t})-Z( Y_t, \mathscr{L}_{Y_t})\Big) \d t.
\end{aligned}
\right.
\end{align}
Since $r_0\in(-\|B\|^{-1},\|B\|^{-1})$, for any $r>0$ there exists a constant $C>1$ such that
\begin{align*}
&\frac 1 C \big(|X_t^{(1)}-Y_t^{(1)}|^2+|X_t^{(2)}-Y_t^{(2)}|^2\big)\\
&\leq \Psi_t:=\frac{r^2}{2}|X_t^{(1)}-Y_t^{(1)}|^2+\frac 1 2 |X_t^{(2)}-Y_t^{(2)}|^2+r r_0\langle X_t^{(1)}-Y_t^{(1)}, B(X_t^{(2)}-Y_t^{(2)}) \rangle\\
&\leq C \big(|X_t^{(1)}-Y_t^{(1)}|^2+|X_t^{(2)}-Y_t^{(2)}|^2\big).
\end{align*}
Combining this with (\ref{ODE}) and (D3), we obtain
\begin{align*}
\d \Psi_t &\leq -\theta_1 \big(|X_t^{(1)}-Y_t^{(1)}|^2+|X_t^{(2)}-Y_t^{(2)}|^2\big)+\theta_2 W_2(P_t^{*}\mu_0,P_t^{*}\nu_0)^2,
\end{align*}
by the chain rule, we have
\begin{align*}
\d ( e^{\lambda t}\Psi_t) &\leq e^{\lambda t} \big\{ \lambda \Psi_t -\theta_1 \big(|X_t^{(1)}-Y_t^{(1)}|^2+|X_t^{(2)}-Y_t^{(2)}|^2\big)+\theta_2 W_2(P_t^{*}\mu_0,P_t^{*}\nu_0)^2 \big\}\d t,
\end{align*}
thus we obtain
\begin{align*}
\mathbb{E} \Psi_t &\leq e^{-\lambda t} \mathbb{E} \Psi_0- e^{-\lambda t} \int_0^t e^{\lambda s}(\theta_1-\theta_2-\lambda C) \mathbb{E}\big[ |X_s^{(1)}-Y_s^{(1)}|^2+|X_s^{(2)}-Y_s^{(2)}|^2 \big] \d s,
\end{align*}
we take $\lambda=\frac{\theta_1-\theta_2}{2 C}$ and we obtain
\begin{align*}
\mathbb{E} \Psi_t &\leq e^{-\frac{\theta_1-\theta_2}{2 C} t} \mathbb{E} \Psi_0,
\end{align*}
and we deduce that
\begin{align*}
W_2(P_t^{*}\mu_0,P_t^{*}\nu_0)^2 &\leq \mathbb{E}[ \rho(X_t,Y_t)^2 ]\\
&\leq C e^{-\frac{\theta_1-\theta_2}{2 C} t} W_2(\mu_0,\nu_0)^2 ,
\end{align*}
Consequently, $P_t^{*}$ has a unique invariant probability measure $\bar{\mu}$ such that (\ref{W2SHS}) holds.

Next, let $\mathscr{L}_{\bar{X}_0}=\bar{\mu}$, consider the reference Stochastic Hamiltonian System
for $\bar{X}_t=(\bar{X}_t^{(1)}, \bar{X}_t^{(2)})$ on $\mathbb{R}^{m+d}$:
\begin{align}\label{reSHS}
\left\{
\begin{aligned}
\d \bar{X}_t^{(1)} & =  (A \bar{X}_t^{(1)}+B \bar{X}_t^{(2)}) \d t, \\
\d \bar{X}_t^{(2)} & =  Z( \bar{X}_t, \bar{\mu}) dt+M \d B_t.
\end{aligned}
\right.
\end{align}
According to \cite[Theorem 3.1]{WZ19}, $\bar{L}_t^A\in \textrm{MDP}(I)$ for $I(y)={y^2}/({8 \bar{V}(A)})$. Since $A$ is Lipschitz continuous, by (\ref{W2SHS}), we can find some small $\delta> 0$ such that (\ref{EeN}) holds. Therefore, the proof is finished by Theorem \ref{ThEeN}.

\paragraph{Acknowledgement.} The authors would like to thank Professor Feng-Yu Wang for supervision.

\beg{thebibliography}{99}

\bibitem{B88} P. A. Baldi, \emph{Large deviations and stochastic homogenisation}, Ann. Mat. Pura Appl. 151(1988), 161--177.

\bibitem{BM78} A. A. Borovkov, A. A. Mogulskii, \emph{Probabilities of large deviations in topological vector space I}, Siberian Math. J. 19(1978), 697--709.

\bibitem{BM80} A. A. Borovkov, A. A. Mogulskii, \emph{Probabilities of large deviations in topological vector space II}, Siberian Math. J. 21(1980), 12--26.

\bibitem{BWY20} J. Bao, F.-Y. Wang, C. Yuan, \emph{Limit theorems for additive functionals of path-depedent SDEs}, Discrete Contin. Dyn. Syst. 40(2020), 5173--5188.


\bibitem{C91} X. Chen, \emph{The moderate deviations of independent random vectors in a Banach space}, Chinese J. Appl. Probab. Statist. 7(1991), 24--32.

\bibitem{CDR17} P. Cattiaux, P. Dai Pra, S. Roelly, \emph{A constructive approach to a class of ergodic HJB equatons with unbounded and nonsmooth cost}, SIAM J. Control Optim. 47(2008), 2598--2615.

\bibitem{DV75} M. D. Donsker, S. R. S. Varadhan, \emph{Asymptotic evaluation of certain Markov process expectations for large time, I-IV}, Comm. Pure Appl. Math. 28(1975), 1--47, 279--301; 29(1976), 389--461; 36(1983), 183--212.

\bibitem{DZ98} A. Dembo, O. Zeitouni, \emph{Large Deviations Techniques and Applications}, Second Edition, Springer, New York. 1998.

\bibitem{Gao17} F. Gao, \emph{Long time asymptotics of unbounded additive functionals of Markov processes}, Electron. J. Probab. 22(2017), 1--21.

\bibitem{HRW}  X. Huang, P. Ren, F.-Y. Wang, \emph{Distribution Dependent Stochastic Differential Equation}, arXiv:2012.13656.

\bibitem{IN64} K. It\^{o}, M. Nisio, \emph{On stationary solutions of a stochastic differential equation}, J. Math.Kyoto Univ. 4(1964), 1--75.

\bibitem{KM03} I. Kontoyiannis, S. P. Meyn, \emph{Spectral theory and limit theorems for geometrically ergodic Markov processes}, Ann. Appl. Probab. 13(2003), 304--362.

\bibitem{RW20} P. Ren, F.-Y. Wang, \emph{Donsker-Varadhan Large Deviations for Path-Distribution Dependent SPDEs}, arXiv:2002.08652.

\bibitem{RWW06} M. R\"{o}ckner, F.-Y. Wang, L. Wu, \emph{Large deviations for stochastic generalized porous media equations}, Stoch. Proc. Appl. 116(2006), 1677--1689.

\bibitem{SV79} D. W. Stroock, S. R. S. Varadhan, \emph{Multidimensional Diffusion Processes}, Springer, New York. 1979.

\bibitem{Wfy11} F.-Y. Wang, \emph{Harnack inequality for SDE with multiplicative noise and extension to Neumann semigroup on nonconvex mainfolds}, Ann. Probab. 39(2011), 1449--1467.

\bibitem{Wfy17} F.-Y. Wang, \emph{Hypercontractivity and applications for stochastic Hamiltonian systems}, J. Funct. Anal. 272(2017), 5360--5383.

\bibitem{Wang18} F.-Y. Wang, \emph{Distribution dependent SDEs for Landau type equations.} Stoch. Proc. Appl. 128(2018), 595--621.

\bibitem{WZ19} F.-Y. Wang, Y. Zhang, \emph{Application of Harnack inequality to long time asymptotics of Markov processes(in Chinese)}, Sci. Sin. Math. 49(2019), 505--516.

\bibitem{Wu95} L. Wu, \emph{Moderate deviations of dependent random variables related to CLT}, Ann. Probab. 23(1995), 420--445.

\bibitem{Wu00} L. Wu, \emph{Uniformly integrable operators and large deviations for Markov processes}, J. Funct. Anal. 172(2000), 301--376.

\end{thebibliography}

\end{document}